\documentclass[11pt]{amsart}%
\usepackage{amsfonts}
\usepackage{graphicx}
\usepackage{amscd}
\usepackage{amsmath}%
\usepackage{amssymb}

\theoremstyle{plain}
\newtheorem{theorem}{Theorem}[section]
\newtheorem{corollary}{Corollary}[section]
\newtheorem{definition}{Definition}[section]
\newtheorem{lemma}{Lemma}[section]
\newtheorem{proposition}{Proposition}[section]

\theoremstyle{definition}

\newtheorem{remark}{Remark}[section]

\numberwithin{equation}{section}

\begin{document}
\title[Distributional Versions of Littlewood's Theorem]{Distributional Versions of Littlewood's Tauberian Theorem}
\author{Ricardo Estrada}
\address{Department of Mathematics, Louisiana State University, Baton Rouge, LA 70803, USA.}
\email{restrada@math.lsu.edu}
\author{Jasson Vindas}
\address{Department of Mathematics\\ Ghent University\\ Krijgslaan 281 Gebouw S22\\ B 9000 Gent\\ Belgium}
\email{jvindas@cage.Ugent.be}
\thanks{R. Estrada gratefully acknowledges support from NSF, through grant number 0968448.}
\thanks{J. Vindas gratefully acknowledges support by a Postdoctoral Fellowship of the Research Foundation--Flanders (FWO, Belgium)}
\begin{abstract} We provide several general versions of Littlewood's Tauberian theorem. These versions are applicable to Laplace transforms of Schwartz distributions. We apply these Tauberian results to deduce a number of Tauberian theorems for power series where Ces\`{a}ro summability follows from Abel summability. We also use our general results to give a new simple proof of the classical Littlewood one-sided Tauberian theorem for power series.
\end{abstract}

\maketitle

\section{Introduction}
\label{il}

A century ago, Littlewood obtained his celebrated extension of Tauber's theorem \cite{tauber,littlewood}. Littlewood's Tauberian theorem states that if the series $\sum_{n=0}^{\infty} c_n$ is Abel summable to the number $a$, namely, the power series $\sum^{\infty}_{n=0}c_nr^n$ has radius of convergence at least $1$ and 
\begin{equation}
\label{ileq1}
\lim_{r\rightarrow 1^-}\sum^{\infty}_{n=0}r^nc_n=a \ ,
\end{equation}  
and if the Tauberian hypothesis 
\begin{equation}
\label{ileq2}
c_n=O\left(\frac{1}{n}\right)
\end{equation}
is satisfied, then the series is actually convergent,  $\sum_{n=0}^{\infty} c_n=a$.

The result was later strengthened by Hardy and Littlewood in \cite{hardy-littlewood1,hardy-littlewood2} to an one-sided version.  They showed that the condition (\ref{ileq2}) can be relaxed to the weaker one $nc_n=O_L\left(1\right)$, i.e., there exists $C>0$ such that 
\begin{equation*}
-C<nc_n \ .
\end{equation*} 

The aim of this article is to provide several distributional versions of this Hardy-Littlewood Tauberian theorem, our versions shall include it as a particular case. Our general results are in terms of Laplace transforms of distributions, and they have interesting consequences when applied to Stieltjes integrals and numerical series. In particular, we shall provide various Tauberian theorems where the conclusion is Ces\`{a}ro (or Riesz) summability rather than convergence. 

We state a sample of our results. The ensuing theorem will be derived in Section \ref{l1ns} (cf. Corollary \ref{lc5}). 
In order to state it, we need to introduce some notation. We shall write 
\begin{equation*}
b_n=O_{\textnormal{L}}(1) \ \ \ (\mathrm{C},m)
\end{equation*}
if the Ces\`{a}ro means of order $m\geq1$ of a sequence $\left\{b_n\right\}^{\infty}_{n=0}$ (not to be confused with the ones of a series) are bounded from below, namely, there is a constant $K>0$ such that
\begin{equation*}-K<\frac{m!}{n^m}\sum^{n}_{k=0}
\binom
{k+m-1}{
m-1}
b_{n-k}.
\end{equation*} 
 
\begin{theorem}
\label{ilth1}
If $\sum_{n=0}^{\infty}c_{n}=a\ $ $(\mathrm{A})$,  then the Tauberian condition 
\begin{equation}
\label{ileq5}
nc_n=O_{\textnormal{L}}(1) \ \ \ (\mathrm{C},m) \ .
\end{equation}
 implies the $(\mathrm{C},m)$ summability of the series, $\sum_{n=0}^{\infty}c_{n}=a\ $ $(\mathrm{C},m)$.
\end{theorem}

Tauberian theorems in which Ces\`{a}ro summability follows from Abel summability have a long tradition, which goes back to Hardy and Littlewood \cite{littlewood,hardy-littlewood1931}. Such results have also received much attention in recent times, e.g., \cite{canak-e-t,pati2005}. Actually, Pati and \c{C}anak et al have made extensive use of  
Tauberian conditions involving the Ces\`{a}ro means of $nc_{n}$, such as (\ref{ileq5}), in the study of Tauberian theorems for the so called $(\mathrm{A})(\mathrm{C},\alpha)$ summability.

We would like to point out that there is an extensive literature in Tauberian theorems for Schwartz distributions, an overview can be found in \cite{p-s-v, vladimirov-d-z1}. Extensions of the Wiener Tauberian theorem have been obtained in \cite{peetre,pilipovic-stankovic2,pilipovic-stankovic3} (cf. \cite{p-s-v}). Recent applications to the theory of Fourier and conjugate series are considered in \cite{estrada-vindasFC}. We also mention that the results of this article are closely related to those from \cite{estrada-vindasT,vindas-estrada4}, though with a different approach.

For future purposes, it is convenient to restate Hardy-Littlewood theorem in a form which is invariant under addition of terms of the form $n^{-1}M$. Set $b_{0}=c_{0}$, write $b_n=c_n+C/n$, for $n>0$, and $r=e^{-y}$.  Then (\ref{ileq1}) transforms into
\begin{equation*}
\sum^{\infty}_{n=0}b_ne^{-ny}=-C\log(1-e^{-y})+\sum^{\infty}_{n=0}c_ne^{-ny} 
=a+C\log\left(\frac{1}{y}\right)+o(1)\ ,
\end{equation*}
while the convergence conclusion translates into $\sum^{N}_{n=0}b_n=a+C\gamma+C\log N+o(1)$, $N\rightarrow\infty$, where $\gamma$ is the Euler gamma constant. 
Therefore, Hardy-Littlewood theorem might be formulated as follows.

\begin{theorem}
\label{lth1}
Let $\sum^{\infty}_{n=0}c_ne^{-ny}$ be convergent for $y>0$.  Suppose that 
\begin{equation}
\label{ileq4}
\lim_{y\rightarrow0^+}\sum^{\infty}_{n=0}c_ne^{-ny}-b\log\left(\frac{1}{y}\right)=a\ .
\end{equation}
Then, the Tauberian hypothesis $nc_{n}=O_{\textnormal{L}}(1)$
implies that 
\begin{equation}
\label{ileq6}
\sum^{N}_{n=0}c_n=a+b\gamma+b\log N+o(1)\ ,  \ \ \ N\rightarrow\infty \ .
\end{equation}
\end{theorem}

Theorem \ref{lth1} is precisely the form of Littlewood's theorem which we will generalize to distributions. The plan of this article is as follows. In Section \ref{lp} we explain the notions from distribution theory to be used in this paper. Section \ref{l2} provides a two-sided distributional version of Littlewood's theorem. We shall use such a version to produce a simple proof of the classical Littlewood one-sided theorem. We give a one-sided Tauberian theorem for Laplace transforms of distributions in Section \ref{l1} and then discuss some applications to Stieltjes integrals and numerical series; as an example we extend a classical theorem of Sz\'{a}sz \cite{szasz3}.  

\section{Preliminaries and Notation}
\label{lp}
\subsection{Distributions} The spaces of test functions and distributions $\mathcal{D}(\mathbb{R})$, $\mathcal{S}(\mathbb{R})$, $\mathcal{D'}(\mathbb{R})$, and $\mathcal{S'}(\mathbb{R})$ are well known for most analysts, we refer to \cite{schwartz,vladimirovbook} for their properties. We denote by $\mathcal{S}[0,\infty)$ the space of restrictions of test functions from $\mathcal{S}(\mathbb{R})$ to the interval $[0,\infty)$; its dual space $\mathcal{S}'[0,\infty)$ is canonically isomorphic \cite{vladimirovbook} to the subspace of distributions from $\mathcal{S'}(\mathbb{R})$ having supports in $[0,\infty)$.

We shall employ several special distributions, we follow the notation exactly as in \cite{estrada-kanwal2}. For instance, $\delta$ is as usual the Dirac delta, $H$ is the Heaviside function, i.e., the characteristic function of $[0,\infty)$, the distributions $x_{+}^{\beta-1}$ are simply given by $x^{\beta-1}H(x)$
whenever $\Re e \:\beta>0$, and $\operatorname*{Pf}(H(x)/x)$ is defined via Hadamard finite part regularization, i.e.,
$$
\left\langle \operatorname*{Pf} \left(\frac{H(x)}{x}\right),\phi(x)\right\rangle= \int_{0}^{1}\frac{\phi(x)-\phi(0)}{x}\mathrm{d}x+\int_{1}^{\infty}\frac{\phi(x)}{x}\mathrm{d}x\ .
$$

\subsection{Ces\`{a}ro Limits}
\label{lcl}
We refer to \cite{estrada4,estrada-kanwal2} for the Ces\`{a}ro behavior of distributions. We will only consider Ces\`{a}ro limits. Given $f\in\mathcal{D'}(\mathbb{R})$ with support bounded at the left, we write

\begin{equation}
\label{lpeq1}
\lim_{x\to\infty}f(x)=\ell \ \ \ (\mathrm{C},m)
\end{equation}
if $f^{(-m)}$, the $m$-primitive of $f$ with support bounded at the left, is an ordinary function for large arguments and
$$
f^{(-m)}(x)\sim \frac{\ell x^{m}}{m!}\ , \ \ \ x\to\infty\ . 
$$ 
Observe that $f^{(-m)}$ is given by the convolution \cite{vladimirovbook}
\begin{equation*}
f^{(-m)}=f \ast \frac{x_{+}^{m-1}}{(m-1)!}\ .
\end{equation*}
If we do not want to make any reference to $m$ in (\ref{lpeq1}), we simply write $(\mathrm{C})$. In the special case when $f=s$ is a function of local bounded variation with $s(x)=0$ for $x<0$, then (\ref{lpeq1}) reads as
$$
\lim_{x\rightarrow\infty}\int^{x}_{0}\left(1-\frac{t}{x}\right)^m \mathrm{d}s(t)=s(0)+\ell\ .
$$
Thus, if $s$ is given by the partial sums of a series $\sum_{n=0}^{\infty}c_{n}$, this notion amounts to the same as $\sum_{n=0}^{\infty}c_{n}=\ell$ $(\mathrm{C},m)$, as shown by the equivalence between Ces\`{a}ro and Riesz summability \cite{hardy,ingham}.
\subsection{Laplace Transforms} Let $f\in\mathcal{D'}(\mathbb{R})$ be supported in $[0,\infty)$, it is said to be
Laplace transformable \cite{schwartz} on $\Re e\:z>0$ if $e^{-yx}f \in\mathcal{S}'(\mathbb{R})$ is a tempered distribution for all $y>0$. In such a case its Laplace transform is well defined on the half-plane $\Re e\:z>0$ and it is given by the evaluation
$$
\mathcal{L}\left\{f;z\right\}=\left\langle f(x),e^{-zx}\right\rangle\ .
$$
If $f=s$ is a function of local bounded variation, then one readily verifies that it is Laplace transformable on $\Re e\:z>0$ in the distributional sense if and only if 
\begin{equation}
\label{2talli}
\mathcal{L}\left\{\mathrm{d}s;y\right\}:=\int_{0}^{\infty}e^{-yx}\mathrm{d}s(x) \ \ \ (\mathrm{C}) \ \ \ \text{exists for each}\ y>0\ ,
\end{equation}
Thus, Laplace transformability in this context is much more general than the mere existence of Laplace-Stieltjes improper integrals. Observe also that the order of $(\mathrm{C})$ summability might quickly change in (\ref{2talli}) with each $y$.

\subsection{Distributional Asymptotics} We shall make use of the theory of asymptotic expansions of distributions, explained for example in \cite{estrada-kanwal2,p-s-v,vindas1,vindas-pilipovic1}. For instance, let $f,g_{1},g_{2}\in\mathcal{S}'(\mathbb{R})$ and let $c_{1}$ and $c_{2}$ be two positive functions such that $c_{2}(\lambda)=o(c_{1}(\lambda))$, $\lambda\to\infty$. The asymptotic formula 
$$
f(\lambda x)= c_{1}(\lambda)g_{1}(x)+c_{2}(\lambda)g_{2}(x)+o(c_{2}(\lambda)) \ \ \ \text{as} \ \lambda\rightarrow\infty \ \ \text{in} \ \mathcal{S}'(\mathbb{R}) \ ,
$$
is interpreted in the distributional sense, namely, it means that for all test functions $\phi\in\mathcal{S}(\mathbb{R})$
$$
\left\langle f(\lambda x),\phi(x)\right\rangle= c_{1}(\lambda)\left\langle g_{1}(x),\phi(x)\right\rangle+c_{2}(\lambda)\left\langle g_{2}(x),\phi(x)\right\rangle+o(c_{2}(\lambda)) \ .
$$
\section{Distributional Littlewood two-sided Tauberian Theorem}
We want to find a distributional analog to (\ref{ileq6}).  Set 
$s(x)=\sum_{n<x}c_n,$
then (\ref{ileq6}) gives $s(x)=a+b\gamma+b\log x+o(1).$
It is now easy to prove \cite[Lem 3.9.2]{estrada-kanwal2} that the previous ordinary expansion implies the distributional expansion 
\begin{equation*}
s(\lambda x)=(a+b\gamma)H(x)+bH(x)\log(\lambda x)+o(1) \ \ \ \text{as} \ \lambda\rightarrow\infty \ \ \text{in} \ \mathcal{S}'(\mathbb{R}) \ ;
\end{equation*}
differentiating \cite{estrada-kanwal2}, we obtain
\begin{equation*}
s'(\lambda x)=(a+b\gamma+ b\log\lambda)\frac{\delta(x)}{\lambda}+\frac{b}{\lambda}\operatorname*{Pf}\left(\frac{H(x)}{x}\right)+o\left(\frac{1}{\lambda}\right) \ \ \ \text{as} \ \lambda\rightarrow\infty \ \ \text{in} \ \mathcal{S'}(\mathbb{R}) \ .
\end{equation*}

The above distributional asymptotic relation is the one which we will mostly study in this article. In Subsection \ref{la} we give an Abelian theorem related to it. We give a two-sided Tauberian converse in Subsection \ref{l2} that will be used to produce a new proof of Hardy-Littlewood theorem in the form of Theorem \ref{lth1}. The study of more general one-sided Tauberian conditions will be postponed to Section \ref{l1}.

\subsection{The Abelian Theorem}
\label{la}

We begin with the following Abelian theorem for Laplace transforms of distributions.

\begin{theorem}
\label{lth2}
Let $g\in\mathcal{S'}(\mathbb{R})$ be supported in $[0,\infty)$ and have the distributional asymptotic behavior 
\begin{equation}
\label{laeq1}
g(\lambda x)=a\frac{\delta(x)}{\lambda}+b\frac{\log\lambda}{\lambda}\delta(x)+\frac{b}{\lambda}\operatorname*{Pf}\left(\frac{H(x)}{x}\right)+o(1) \ \ \ \text{as} \ \lambda\rightarrow\infty 
\end{equation}
in $\mathcal{S'}(\mathbb{R})$.
Then,
\begin{equation}
\label{laeq2}
\mathcal{L}\left\{g;y\right\}=a-b\gamma+b\log\left(\frac{1}{y}\right)+o(1)\ , \ \ \ y\rightarrow0^+\ .
\end{equation}
\end{theorem}

\begin{proof}
Writing $\lambda=y^{-1}$, we have, as $\lambda\to\infty$,
\begin{align*}
\mathcal{L}\left\{g;\lambda^{-1}\right\}&
=\lambda\left\langle g(\lambda x),e^{-x}\right\rangle
\\
&
=(a+b\log\lambda)\left\langle \delta(x),e^{-x}\right\rangle+b\left\langle \operatorname*{Pf}\left(\frac{H(x)}{x}\right),e^{-x}\right\rangle+o(1)
\\
&
=a+b\log\lambda+b\:\mathrm{F.p.}\int^{\infty}_{0}\frac{e^{-x}}{x}\:\mathrm{d}x+o(1)
\\
&
=a+b\log\lambda-b\gamma+o(1)\ .
\end{align*}
\end{proof}

\begin{corollary}
\label{lc1}
Let $s$ be a function of local bonded variation such that $s(x)=0$ for $x\leq0$.
If
\begin{equation}
\label{laeq3}
\lim_{x\rightarrow\infty}\left(s(x)-b\log x\right)=a \ \ \ (\mathrm{C})\ ,
\end{equation}
then, $\mathcal{L}\left\{\mathrm{d}s;y\right\}:=\int^{\infty}_{0}e^{-yx}\mathrm{d}s(x)$ is $(\mathrm{C})$ summable for each $y>0$, and 
\begin{equation}
\label{laeq4}
\mathcal{L}\left\{\mathrm{d}s;y\right\}=a-b\gamma+b\log\left(\frac{1}{y}\right)+o(1) \ , \ \ \ y\rightarrow0^+\ .
\end{equation}
\end{corollary}

\begin{proof}
Set $g=s'$. The Ces\`{a}ro limit (\ref{laeq3}) implies \cite{estrada-kanwal2} that  $s(\lambda x)=aH(x)+bH(x)\log (\lambda x)+o(1)$ as $\lambda\rightarrow\infty$ in $\mathcal{S'}(\mathbb{R})$. Differentiating, we conclude that $g$ satisfies (\ref{laeq1}), and so, by Theorem \ref{lth2}, we deduce (\ref{laeq4}). 
\end{proof}

In particular if we consider $s(x)=\sum_{n< x}c_{n}$, we obtain that (\ref{ileq6}) implies (\ref{ileq4}), the Abelian counterpart of Theorem \ref{lth1}.

We end this subsection by pointing out that (\ref{laeq1}) is the most general asymptotic separation of variables we could have in the situation that we are studying. The proof of the following proposition follows from the general results from \cite{estrada-kanwal93}.

\begin{proposition}
\label{lp1}
Let $g\in\mathcal{S'}(\mathbb{R})$ be supported in $[0,\infty)$.  If there are $g_1,g_2\in\mathcal{S'}(\mathbb{R})$ such that 
\begin{equation*}
g(\lambda x)=\frac{\log\lambda}{\lambda}g_1(x)+\frac{1}{\lambda}g_2(x)+o\left(\frac{1}{\lambda}\right)\ \ \ \text{as} \ \lambda\rightarrow\infty \
\text{in} \ \mathcal{S'}(\mathbb{R})\ ,
\end{equation*}
then $g_1(x)=b\delta(x)$ and $g_2(x)=a\delta(x)+b\operatorname*{Pf}(H(x)/x)$, for some constants $a$ and $b$. Consequently, $g$ has the distributional asymptotic behavior (\ref{laeq1}).
\end{proposition}

\subsection{Functions and the Distributional Asymptotics (\ref{laeq1})}
\label{lf}

We shall prove that if $s$ is non-decreasing and $s'$ has the distributional asymptotic behavior (\ref{laeq1}), then one recovers the asymptotic behavior (\ref{laeq3}) in the ordinary sense.

\begin{proposition}
\label{lp2}
Let $s\in L^{1}_{\textnormal{loc}}(\mathbb{R})$ be supported in $[0,\infty)$.  If there exist $A,B>0$ such that $s(x)+A\log x$ is non-decreasing on the interval $[B,\infty)$ and 
\begin{equation}
\label{lfeq1}
s'(\lambda x)=a\frac{\delta(x)}{\lambda}+b\frac{\log\lambda}{\lambda}\delta(x)+\frac{b}{\lambda}\operatorname*{Pf}\left(\frac{H(x)}{x}\right) +o\left(\frac{1}{\lambda}\right) \ ,
\end{equation}
as $\lambda\rightarrow\infty$ in $\mathcal{S'}(\mathbb{R})$, then
\begin{equation}
\label{lfeq2}
\lim_{x\rightarrow\infty}(s(x)-b\log x)=a \ .
\end{equation}
\end{proposition}

\begin{proof}
We may assume that $s(0)=0$ and that $s$ is non-decreasing on the whole $\mathbb{R}$. Let $\varepsilon$ be an arbitrary small number.  Pick $\phi_1,\phi_2\in\mathcal{D}(\mathbb{R})$ such that $0\leq\phi_j\leq 1$, $\operatorname*{supp}\phi_2\subseteq[-1,1+\varepsilon]$, $\phi_2(x)=1$ for $x\in[0,1]$, $\operatorname*{supp}\phi_1\subseteq[-1,1]$ and $\phi_1(x)=1$ for $x\in[-1,1-\varepsilon]$.  Evaluating (\ref{lfeq1}) at $\phi_2$ we have
\begin{align*}
\limsup_{\lambda\rightarrow\infty}\left(s(\lambda)-b\log\lambda\right)&
\leq\lim_{\lambda\rightarrow\infty}\left(\int^{\infty}_{0}\phi_2\left(\frac{x}{\lambda}\right)\mathrm{d}s(x)-b\log\lambda\right)
\\
&
=a+b\:\mathrm{F.p}\int^{\varepsilon+1}_{0}\phi_2(x)\frac{\mathrm{d}x}{x}=a+b\int^{\varepsilon+1}_{1}\phi_2(x)\frac{\mathrm{d}x}{x}
\\
&
\leq a+b\varepsilon\ .
\end{align*}
Likewise, evaluation at $\phi_1$ yields 
\begin{align*}
\liminf_{\lambda\rightarrow\infty}\left(s(\lambda)-b\log\lambda\right)
&\geq a+b\:\mathrm{F.p}\int^{\infty}_{0}\phi_1(x)\frac{\mathrm{d}x}{x}=a+b\int^{1}_{1-\varepsilon}\frac{\phi_1(x)-1}{x}\mathrm{d}x
\\
&
\geq a+b\log(1-\varepsilon)\ .
\end{align*}
Since $\varepsilon$ was arbitrary, we conclude (\ref{lfeq2}).
\end{proof}

\subsection{Distributional two-sided Tauberian Theorem}
\label{l2}

We now show our first distributional version of Littlewood Tauberian theorem.  It is the Tauberian converse of Theorem \ref{lth2}.  Since we use the big $O$ symbol in the Tauberian hypothesis, we denominate it a two-sided Tauberian theorem.

\begin{theorem}
\label{lth3}
Let $g\in\mathcal{S'}(\mathbb{R})$ be supported on $[0,\infty)$.  Suppose that, as $y\rightarrow0^+$, 
\begin{equation}
\label{l2eq1}
\mathcal{L}\left\{g;y\right\}=a+b\log\left(\frac{1}{y}\right)+o(1) \ .
\end{equation}
Then, the Tauberian hypothesis 
\begin{equation}
\label{l2eq2}
g(\lambda x)-b\log\lambda\frac{\delta(x)}{\lambda}=O\left(\frac{1}{\lambda}\right) \ ,
\end{equation}
implies the distributional asymptotic behavior
\begin{equation}
\label{l2eq3}
g(\lambda x)=(a+b\gamma)\frac{\delta(x)}{\lambda}+b\frac{\log\lambda}{\lambda}\delta(x)+\frac{b}{\lambda}\operatorname*{Pf}\left(\frac{H(x)}{x}\right)+o(1) \ .
\end{equation}
\end{theorem}

\begin{proof}
Let $g_{\lambda}(x)=\lambda g(\lambda x)-b\log\lambda\delta(x)$.  Let $\mathfrak{B}$ be the linear span of $\left\{e^{-\tau x}\right\}_{\tau\in\mathbb{R}}$.  Observe that $\mathfrak{B}$ is dense in $\mathcal{S}[0,\infty)$, due to the Hahn-Banach theorem and the fact that the Laplace transform is injective.  Next, we verify that 
\begin{equation*}
\lim_{\lambda\rightarrow\infty}\left\langle g_{\lambda}(x),\phi(x)\right\rangle=\left\langle (a+b\gamma)\delta(x)+b\operatorname*{Pf}\left(\frac{H(x)}{x}\right),\phi(x)\right\rangle \ ,  \ \ \ \phi\in\mathfrak{B}\ .
\end{equation*}
Indeed, it is enough for $\phi(x)=e^{-\tau x}$; by (\ref{l2eq1}), as $\lambda\rightarrow\infty$, 
\begin{align*}
\left\langle g_{\lambda}(x),e^{-\tau x}\right\rangle & =\mathcal{L}\left\{g,\frac{\tau}{\lambda}\right\}-b\log\lambda
=a+b\log\left(\frac{\lambda}{\tau}\right)-b\log\lambda+o(1)\\
&
=\left\langle (a+b\gamma)\delta(x)+b\operatorname*{Pf}\left(\frac{H(x)}{x}\right),e^{-\tau x}\right\rangle+o(1).
\end{align*}
Now, the Tauberian hypothesis (\ref{l2eq2}) implies that $\left\{g_{\lambda}\right\}_{\lambda\in[1,\infty)}$ is weakly bounded in $\mathcal{S}'[0,\infty)$, and so, by the Banach-Steinhaus theorem, it is equicontinuous.  Since an equicontinuous family of linear functionals converging over a dense subset must be convergent, we obtain that 
$$
\lim_{\lambda\rightarrow\infty}g_{\lambda}(x)=(a+b\gamma)\delta(x)+b\operatorname*{Pf}(H(x)/x)\ \ \  \mbox{in } \mathcal{S'}(\mathbb{R}),
$$
which is precisely (\ref{l2eq3}).
\end{proof}

\subsection{Classical Littlewood's One-sided Theorem}\label{lct}
Let us show how our two-sided Tauberian theorem can be used to give a simple proof of Hardy-Littlewood theorem in the form of Theorem \ref{lth1}.  We actually give a more general result for Stieltjes integrals.

\begin{remark} In many proofs of Littlewood's one-sided theorem, such as the one based in Wiener's method, one needs to establish first the boundedness of $s(x)=\sum_{n<x}c_{n}$, which is not an easy task \cite{hardy,korevaarbook,wiener}. The method that we develop in the proof of Theorem \ref{lc2} rather estimates the second order Riesz means, which turns out to be much simpler.
\end{remark}

\begin{theorem}
\label{lc2}
Let $s$ be of local bounded variation and supported in $[0,\infty)$. Suppose that 
(\ref{2talli}) holds. Furthermore, assume that there exist $A,B>0$ such that $s(x)+A\log x$ is non-decreasing on $[B,\infty)$.  Then
\begin{equation}
\label{l2eq4}
\mathcal{L}\left\{\mathrm{d}s;y\right\}=a+b\log\left(\frac{1}{y}\right)+o(1)\ , \ \ \ y\rightarrow 0^+,
\end{equation}
if and only if 
\begin{equation}
\label{l2eq5}
s(x)=s(0)+a+b\gamma+b\log x+o(1)\ , \ \ \ x\rightarrow\infty \ .
\end{equation}
\end{theorem}

\begin{proof}
One direction is implied by Corollary \ref{lc1}.  For the other part, we may assume that $s(0)=0$ and that $s$ is non-decreasing over the whole real line. Consider the second order primitive $s^{(-2)}(x)=\int^{x}_{0}(x-t)s(t)\mathrm{d}t$. Our strategy will be to show

\begin{equation}
\label{l2eq6}
s^{(-2)}(x)=b\frac{x^{2}}{2}\log x+O(x^{2}) \ .
\end{equation}
Suppose for the moment that we were able to show this claim. Let us deduce (\ref{l2eq5}) from (\ref{l2eq6}).  By (\ref{l2eq6}), we obtain the distributional relation
\begin{equation*}
s^{(-2)}(\lambda x)=b\frac{(\lambda x)^{2}}{2}H(x)\log\lambda +O(\lambda^{2})\ ,
\end{equation*}
in $\mathcal{S'}(\mathbb{R})$.  Differentiating three times, $s'(\lambda x)-b\lambda^{-1}\log\lambda\delta(x)=O\left(1/ \lambda\right)$ in $\mathcal{S'}(\mathbb{R})$.
Applying Theorem \ref{lth3} to $g=s'$, we obtain that $s'$ has the asymptotic behavior (\ref{l2eq3}).  Thus, Proposition \ref{lp2} yields (\ref{l2eq6}). 

It then remains to show (\ref{l2eq6}). We start by looking at $s^{-1}(x)=\int_{0}^{x}s(t)\mathrm{d}t$. Since $1-t\leq e^{-t}$, we have the easy upper estimate
\begin{equation}
\label{l2eq7}
\frac{s^{(-1)}(x)}{x}=\int^{x}_{0}\left(1-\frac{t}{x}\right)\mathrm{d}s(t)\leq\int^{\infty}_{0}e^{-\frac{t}{x}}\mathrm{d}s(t)=b\log x+O_{\textnormal{R}}(1) \ .
\end{equation} 
Notice that (\ref{l2eq7}) yields the upper estimate in (\ref{l2eq6}). Next, define $S(x)=bx\log x-s^{(-1)}(x)+Cx $, where the constant $C>0$ is chosen so large that $S(x)>0$ for all $x>0$. Observe now that the lower estimate in (\ref{l2eq6}) would immediately follow if we show 
$$
\int_{0}^{x}S(t)\mathrm{d}t=O(x^{2})\ .
$$  
Finally, because of (\ref{l2eq4}), we have that
$$
\lim_{y\to0^{+}}y^{2}\int_{0}^{\infty}S(t)e^{-yt}\mathrm{d}t=b-\gamma b-a +C\ ,$$
and hence  
\begin{equation*}
\int_{0}^{x}S(t)\mathrm{d}t\leq e\int_{0}^{x}S(t)e^{-\frac{t}{x}}\mathrm{d}t= O(x^{2})\ .
\end{equation*}
The claim has been established and this completes the proof.
\end{proof}

\section{Littlewood One-sided Tauberian Theorems}
\label{l1}

We want one-sided generalizations of Theorem \ref{lc2} in which the conclusion is Ces\`{a}ro limits. The generalization is in terms of Ces\`{a}ro one-sided boundedness as explained in the next subsection. We shall show below first a Tauberian theorem for Laplace transforms of distributions. In Subsection \ref{l1si} we study Stieltjes integrals and generalize a result of Sz\'{a}sz \cite{szasz3}. Finally, we give applications to numerical series in Subsection \ref{l1ns}; in particular, we prove Theorem \ref{ilth1} . 

\subsection{Distributional Littlewood One-sided Tauberian Theorem}
\label{l1d}
For the distributional generalization, let us rewrite the Tauberian hypothesis of Theorem \ref{lc2} is a more suitable way for our purposes.  Recall a distribution $g$ is said to be \emph{non-negative} on an interval $(B_{1},B_{2})$ if it consides with a non-negative measure on that interval.  In such case we may write $g(x)\geq 0$ on $(B_{1},B_{2})$.   With this notation the Tauberian hypothesis of Theorem \ref{lc2} becomes $s'(x)+A/x\geq 0$ on $(B,\infty)$, for some $A,B>0$ or, multiplying by $x$,
$xs'(x)=O_{\textnormal{L}}(1)$ on $(B,\infty)$.
We can also generalize these ideas by using the symbol $O_{\textnormal{L}}(1)$ in the Ces\`{a}ro sense.

\begin{definition}
\label{ld1}
Let $g\in\mathcal{D'}(\mathbb{R})$.  Given $m\in\mathbb{N}$, we say that 
\begin{equation*}
g(x)=O_{\textnormal{L}}(1) \ \ \ \left(\mathrm{C}, m\right) \ , \ \ \ x\rightarrow\infty \ ,
\end{equation*}
if there exist $A,B>0$ and a non-negative measure $\mu$ such that $g(x)+A=\mu^{(m)}$ on $(B,\infty)$.
\end{definition}

Definition \ref{ld1} makes possible to give sense to the relation $xf'(x)=O_{\textnormal{L}}(1)$ in the Ces\`{a}ro sense.

We need also to introduce some notation in order to move further.  For each $m\in\mathbb{N}$, let $l_m$ be the $m$-primitive of $H(x)\log x$ with support in $[0,\infty)$.
It can be verified by induction that for $m\geq 1$
\begin{equation}
\label{l1eq2}
l_m(x)=\frac{1}{(m-1)!}\int^{x}_{0}\log(t)(x-t)^{m-1}\mathrm{d}t=\frac{x^{m}_{+}}{m!}\log x-\frac{x^{m}_{+}}{m!}\sum^{m}_{k=1}\frac{1}{k} \ .
\end{equation}

Let $f$ be supported on $[0,\infty)$. We now study the asymptotic behavior $f(x)=a+b\log x+o(1) \ \ \ (\mathrm{C},m)$, which in view of (\ref{l1eq2}) means that \begin{equation}
\label{l1eq3}
f^{(-m)}(x)=b\frac{x^m_{+}}{m!}\log x+\frac{x^{m}_{+}}{m!}\left(a-b\sum^{m}_{k=1}\frac{1}{k}\right)+o(x^m) \ ,
\end{equation}
$x\rightarrow\infty$, in the ordinary sense. The ensuing Tauberian theorem is a natural distributional version of Littlewood's Tauberian  theorem, in the context of Ces\`{a}ro limits.

\begin{theorem}
\label{lth4}
Let $f\in\mathcal{D'}(\mathbb{R})$ be such that $\operatorname*{supp} f\subseteq[0,\infty)$ and let $m\in\mathbb{N}$. Assume that
\begin{equation}
\label{l1eq5}
xf'(x)=O_{\text{L}}(1) \ \ \ (\mathrm{C},m) \ , \ \ \ x\rightarrow\infty \ .
\end{equation}
Suppose that $f$ is Laplace transformable on $\Re e\:z=y>0$.  Then, 
\begin{equation}
\label{l1eq6}
\mathcal{L}\left\{f',y\right\}=a+b\log\left(\frac{1}{y}\right)+o(1)  \ , \ \ \ y\rightarrow0^+ \ . 
\end{equation}
if and only if
\begin{equation}
\label{l1eq7}
\lim_{x\rightarrow\infty}f(x)-b\log x=a+b\gamma \ \ \ (\mathrm{C},m) \ , \ \ \ x\rightarrow\infty \ .
\end{equation}
\end{theorem}

\begin{proof}
We shall show that $f^{(-m)}$ is locally integrable for large arguments and 
\begin{equation*}
f^{(-m)}(x)=(a+b\gamma)\frac{x^m}{m!}+l_m(x)+o(x^m) \ , \ \ \ x\rightarrow\infty \ .
\end{equation*}
Setting $\tau(x)=x^{-m}f^{(-m)}$, the above asymptotic formula is the same as 
\begin{equation}
\label{l1eq8}
\tau(x)=\frac{a}{m!}+\frac{b\gamma}{m!}-\frac{b}{m!}\sum^{m}_{k=1}\frac{1}{k}+\frac{b}{m!}\log x+o(1) \ , \ \ \ x\rightarrow\infty \ .
\end{equation}
By adding a term of the form $AH(x)\log x$ to $f$ and removing a compactly supported distribution, we may assume that $(xf')^{(-m)}$ is a non-negative measure.   It is clear that we can also assume that $f$, and hence $f^{(-m)}$, is zero in a neighborhood of the origin.  Next, it is easy to verify that $(xf')^{(-m)}=xf^{(-m+1)}-mf^{(-m)}$; multiplying by $x^{-m-1}$, we obtain that $\tau'=x^{-m}f^{(-m+1)}-mx^{-m-1}f^{(-m)}$ is a non-negative measure.  We now look at the Laplace transform of $\tau'$.  Set
\begin{equation*}
F(y)=\mathcal{L}\left\{f';y\right\} \ \ \ \text{and} \ \ \ T(y)=\mathcal{L}\left\{\tau';y\right\} \ .
\end{equation*}
We then have,
\begin{align*}
\left(\frac{T(y)}{y}\right)^{(m)}&=\frac{d^m}{dy^m}\left(\int^{\infty}_{0}\frac{f^{(-m)}(x)}{x^m}e^{-yx}\mathrm{d}x\right)
=(-1)^m\int^{\infty}_{0}f^{(-m)}(x)e^{-yx}\mathrm{d}x\\
&
=(-1)^m\frac{F(y)}{y^{m+1}}
=(-1)^m\frac{a}{y^{m+1}}+(-1)^m\frac{b\log\left(\frac{1}{y}\right)}{y^{m+1}}+o\left(\frac{1}{y^{m+1}}\right) \ ,
\end{align*}
as $y\rightarrow0^+$. Integrating $m$-times the above asymptotic formula and multiplying by $y$ we get 
\begin{equation*}
T(y)=\frac{a}{m!}-\frac{b}{m!}\sum^{m}_{k=1}\frac{1}{k}+\frac{b}{m!}\log\left(\frac{1}{y}\right)+o(1)\ , \ \ \ y\rightarrow0^+ \ .
\end{equation*}
Thus, $\tau$ satisfies the hypothesis of Theorem \ref{lc2}, and (\ref{l1eq8}) follows at once.
\end{proof}

We also have, 

\begin{corollary}
\label{lp3}
Let $f\in\mathcal{D'}(\mathbb{R})$ be supported in $[0,\infty)$.   Suppose that $f'$ has the distributional asymptotic behavior 
\begin{equation*}
f'(\lambda x)=a\frac{\delta(x)}{\lambda}+b\frac{\log\lambda}{\lambda}\delta(x)+\frac{b}{\lambda}\operatorname*{Pf}\left(\frac{H(x)}{x}\right)+o\left(\frac{1}{\lambda}\right) \ .
\end{equation*}
If (\ref{l1eq5}) holds, then $f(x)=a+b\log x+o(1)\ \ \ (\mathrm{C},m),$ $x\rightarrow\infty  .$

\end{corollary}

\begin{proof}
It follows immediately from Theorem \ref{lth2} and Theorem \ref{lth4}.
\end{proof}

\subsection{Stieltjes Integrals}
\label{l1si}

When the distribution $s=f$ is a function of local bounded variation, then (\ref{l1eq5}) can be written as
\begin{equation}
\label{l1eq14}
\int^{x}_{0}\frac{t}{x}\left(1-\frac{t}{x}\right)^{m-1}\mathrm{d}s(t)=O_{\textnormal{L}}(1) \ ,
\end{equation}
for $m\geq 1$. So we obtain at once the ensuing corollary of Theorem \ref{lth4}, it generalizes a classical result of Sz\'{a}sz \cite[Thm.1]{szasz3}.

\begin{corollary}
\label{lc3}
Let $s$ be a function of bounded variation on each finite interval such that $s(x)=0$ for $x<0$.  Furthermore, assume that 
\begin{equation}
\label{l1eq12}
\mathcal{L}\left\{\mathrm{d}s;y\right\}=\int^{\infty}_{0}e^{-yx}\mathrm{d}s(x) \ \ \ (\mathrm{C})
\end{equation}
is summable for each $y>0$. If 
\begin{equation}
\label{l1eq13}
\mathcal{L} \left\{\mathrm{d}s;y\right\}=a+b\log\left(\frac{1}{y}\right)+o(1) \ , \ \ \ y\rightarrow 0^+ \ ,
\end{equation}
then, the Tauberian condition
(\ref{l1eq14}), with $m\geq1$, implies that 
\begin{equation}
\label{l1eq15}
\lim_{x\rightarrow\infty}s(x)-b\log x=s(0)+a+b\gamma \ \ \ (\mathrm{C},m) \ ,
\end{equation}
i.e.,
\begin{equation}
\label{l1eq16}
\lim_{x\rightarrow\infty}\int^{x}_{0}\left(1-\frac{t}{x}\right)^m \mathrm{d}s(t)-b\log x=s(0)+a+b\left(\gamma-\sum^{m}_{k=1}\frac{1}{k}\right) \ .
\end{equation}
\end{corollary}

\begin{remark}
\label{lr1}
Observe that (\ref{l1eq12}) has a general character. We emphasize that it means that for each $y$ there exists $k_{y}\in\mathbb{N}$ such that 
\begin{equation*}
\mathcal{L}\left\{\mathrm{d}s;y\right\}=\int^{\infty}_{0}e^{-yx}\mathrm{d}s(x) \ \ \ (\mathrm{C},k_{y})\ ,
\end{equation*}
and the $k_{y}$ is allowed to become arbitrarily large as $y$ decreases to 0.
\end{remark}
\begin{remark}
\label{lr2}
Integration by parts in (\ref{l1eq14}) shows that it is equivalent to
\begin{equation*}
\int^{x}_{0}\left(1-\frac{t}{x}\right)^{m-1}\mathrm{d}s(t)=\int^{x}_{0}\left(1-\frac{t}{x}\right)^m \mathrm{d}s(t)+O_{\textnormal{L}}(1) \ .
\end{equation*}
\end{remark}

We can specialize Corollary \ref{lc3} to numerical series and obtain the following result about the Riesz means \cite{chandra-mina} of the series. The meaning of $(\mathrm{R},\left\{\lambda_n\right\})$ in the following theorem is $(\mathrm{R},\left\{\lambda_n\right\},k)$ for some $k$. In particular, $k$ may depend on $y$ in relation (\ref{l1eq18}) below.

\begin{corollary}
\label{lc4}
Let $\left\{\lambda_n\right\}^{\infty}_{n=0}$ be an increasing sequence of non-negative tending to infinity.  Assume that 
\begin{equation}
\label{l1eq18}
F(y)=\sum^{\infty}_{n=0}c_ne^{-y\lambda_n} \ \ \ (\mathrm{R},\left\{\lambda_n\right\}) \ 
\end{equation}
is summable for each $y>0$.  If 
\begin{equation}
\label{l1eq19}
F(y)=a+b\log\left(\frac{1}{y}\right)+o(1) \ , \ \ \ y\rightarrow0^+ \ ,
\end{equation}
then, the Tauberian condition 
\begin{equation}
\label{l1eq20}
\sum_{\lambda_n\leq x}c_n\lambda_n\left(1-\frac{\lambda_n}{x}\right)^{m-1}=O_{\textnormal{L}}(x)\ ,
\end{equation}
implies that 
\begin{equation}
\label{l1eq21}
\lim_{x\rightarrow\infty}\sum_{\lambda_n\leq x}c_n\left(1-\frac{\lambda_n}{x}\right)^m-b\log x=a+b\left(\gamma-\sum^{m}_{k=1}\frac{1}{k}\right) \ .
\end{equation}
\end{corollary}
\subsection{Applications to Ces\`{a}ro Summability of Numerical Series}
\label{l1ns}
We end this article by showing that when $\lambda_n=n$ in Corollary \ref{lc4} then the Riesz means might be replaced everywhere by Ces\`{a}ro means. We begin by observing that (\ref{l1eq18}) gives nothing new for $\lambda_n=n$, that is, it simply reduces to convergence of the power series $\sum_{n=0}^{\infty}c_{n}r^{n}$ for $\left|r\right|<1$.

The Ces\`{a}ro means of order $m\geq1$ of a sequence $\left\{b_n\right\}^{\infty}_{n=0}$ are given by
\begin{equation}
\label{l1eq22}
C_m\left\{b_k;n\right\}:=\frac{m!}{n^m}\sum^{n}_{k=0}
\binom
{k+m-1}{
m-1}
b_{n-k} \ .
\end{equation} 
So, if $\lambda_n=n$, then (\ref{l1eq21}) is equivalent to
$$
\lim_{n\to\infty}C_m\{s_{k}- \log k\:;n\}=a+b\left(\gamma-\sum^{m}_{k=1}\frac{1}{k}\right) \ .
$$
with $s_{k}=\sum_{j=0}^{k}c_{j}$, as shown by the equivalence theorem for Riesz and Ces\`{a}ro summability \cite{hardy,ingham}.

Thus, we only need to show that (\ref{l1eq20}) is implied by Ces\`{a}ro one-sided boundedness in the sense already defined in the Introduction. Recall we write 
\begin{equation}
\label{l1eq23}
b_n=O_{\textnormal{L}}(1) \ \ \ (\mathrm{C},m)
\end{equation}
if $C_m\left\{b_k;n\right\}=O_{\textnormal{L}}(1)$. 
So, we have the following lemma.

\begin{lemma}
\label{ll1}
Let $m\in\mathbb{N}$.  If (\ref{l1eq23}) is satisfied, then
\begin{equation*}
\sum_{n\leq x}b_n\left(1-\frac{n}{x}\right)^{m-1}=O_{\textnormal{L}}(x) \ .
\end{equation*}
\end{lemma}

\begin{proof}
We follow closely the proof of \cite[Thm. 58, p. 113]{hardy} and add new information.  Set $B_m(n)=n^{m}C_m\left\{b_k;n\right\}$.
Write $x=n+\vartheta$ with $0\leq\vartheta<1$ and $T_{m-1}(x)=\sum_{0\leq n\leq x}(n-k+\vartheta)^{m-1}b_k$.
We have to show that 
\begin{equation}
\label{l1eq26}
T_{m-1}(x)=O_{\textnormal{L}}(x^m) \ .
\end{equation}
As in \cite[p. 113]{hardy}, one shows that 
\begin{equation*}
T_{m-1}(x)=\sum^{m-1}_{k=0}p^{m-1}_{k}(\vartheta)B_m(n-k)\ ,
\end{equation*}
where each $p^{m-1}_{k}$ is a polynomial, and they are determined by 
\begin{equation*}
P_{m-1}\left(z,\vartheta\right)=\left(1-z\right)^mz^{-\vartheta}\left(z\frac{d}{dz}\right)^{m-1}\left(\frac{z^{\vartheta}}{1-z}\right)=\sum^{m-1}_{k=0}p^{m-1}_{k}(\vartheta)z^k .
\end{equation*}
Observe that if we show that $p^{m-1}_{k}(\vartheta)\geq0$ for all $\vartheta\in[0,1]$ and $0\leq k\leq m-1$, then (\ref{l1eq26}) would follow immediately.  Let us show the latter.  We proceed by induction over $m$.  The statement is clear for $m=1$ since $p_{0}^{0}(\vartheta)=1$.  Assume it for $m-1$. We then have  
\begin{align*}
P_m(z,\vartheta)&=(1-z)^{m+1}z^{-\vartheta}\left(z\frac{d}{dz}\right)^m\left(\frac{z^{\vartheta}}{1-z}\right)
\\
&
=(1-z)^{m+1}z^{1-\vartheta}\left(\frac{z^{\vartheta}}{(1-z)^m}P_{m-1}(z,\vartheta)\right)'
\\
&
=(1-z)zP'_{m-1}(z,\vartheta)+\left(\vartheta+(m-\vartheta)z\right)\:P_{m-1}(z,\vartheta)
\\
&
=\vartheta p^{m-1}_{0}(\vartheta)+(1-\vartheta)p^{m-1}_{m-1}(\vartheta)z^m
\\
&
\ \ \ +\sum^{m-1}_{k=1}\left(\left(k+\vartheta\right)p^{m-1}_{k}(\vartheta)+\left(m-k+1-\vartheta\right)p^{m-1}_{k-1}(\vartheta)\right)z^k \ ,
\end{align*}
thus, 
$$
p^{m}_{0}(\vartheta)=\vartheta p^{m-1}_{0}(\vartheta), \ \ \ p ^{m}_{m}(\vartheta)=(1-\vartheta)p^{m-1}_{m-1}(\vartheta)
$$ 
and 
$$p^{m}_{k}(\vartheta)=(k+\vartheta)p^{m-1}_{k}(\vartheta)+(m-k+1-\vartheta)p^{m-1}_{k-1}(\vartheta)\ , \ \ \ \mbox{for }1\leq k\leq m-1\ .$$  
Therefore, we clearly have $p^{m}_{k}(\vartheta)\geq 0$ for all $\vartheta\in[0,1]$ and $0\leq k\leq m$.
\end{proof}

On combining Corollary \ref{lc4} and Lemma \ref{ll1}, we obtain the ensuing result. It includes both Theorem \ref{ilth1} and Theorem \ref{lth1}.

\begin{corollary}
\label{lc5}
Suppose that $\sum^{\infty}_{n=0}c_nr^n$ is convergent for  $\left|r\right|<1$ and 
\begin{equation}
\label{l1eq27}
nc_n=O_{\textnormal{L}}(1) \ \ \ (\mathrm{C},m) \ .
\end{equation}
Then
\begin{equation}
\label{l1eq28}
F(r)=a+b\log\left(\frac{1}{1-r}\right)+o(1) \ , \ \ \ r\rightarrow 1^-,
\end{equation}
if and only if
\begin{equation}
\label{l1eq29}
\lim_{N\to\infty}\left(\sum^{N}_{n=0}c_n-b\log N\right)=a+b\gamma \ \ \ (\mathrm{C},m) \ .
\end{equation}
\end{corollary}

\end{document}